\documentclass[12pt,twoside,a4paper,reqno]{amsart}

\usepackage{fullpage}

\usepackage{amsmath,amsthm,amsfonts,amssymb}
\usepackage{amsthm}
\usepackage{graphicx}
\usepackage{hyperref}
\usepackage{color}
\usepackage[T1]{fontenc}
\usepackage[utf8]{inputenc}
\usepackage{xy}
\usepackage{graphics}
\usepackage{mathrsfs}
\usepackage{comment}
\usepackage{xcolor}

\xyoption{all}

\makeatletter
\newtheorem*{rep@theorem}{\rep@title}
\newcommand{\newreptheorem}[2]{%
\newenvironment{rep#1}[1]{%
 \def\rep@title{#2 \ref{##1}}%
 \begin{rep@theorem}}%
 {\end{rep@theorem}}}
\makeatother

\numberwithin{equation}{section}
\theoremstyle{plain}
\newtheorem{theorem}{Theorem}[section]
\newreptheorem{theorem}{Theorem}
\newtheorem{proposition}[theorem]{Proposition}

\newreptheorem{corollary}{Corollary}
\newtheorem{conjecture}[theorem]{Conjecture}
\theoremstyle{definition}
\newtheorem{definition}[theorem]{Definition}

\newtheorem{remark}[theorem]{Remark}

\usepackage{tikz-cd}

\begin{document}
\title{Exploring the Interplay of Semistable Vector Bundles and Their Restrictions on Reducible Curves}
\author{Suhas B N, Praveen Kumar Roy and Amit Kumar Singh}
\address{Chennai Mathematical Institute, H1 SIPCOT IT Park, Siruseri, Kelambakkam 603103, India}
\email{chuchabn@gmail.com}
\address{UM-DAE Centre for Excellence in Basic Sciences, University of Mumbai Santacruz, Mumbai 400098, India}
\email{praveen.roy@cbs.ac.in}
\address{Department of Mathematics, SRM University AP, Amaravati 522240, Andhra Pradesh, India}
\email{amitks.math@gmail.com}
\keywords{Vector bundles, Semi(stability), Comb-like curves, Restriction of vector bundles, Strongly unstable bundles.}
\subjclass[2020]{14H60} 
\maketitle

\begin{abstract} Let $C$ be a comb-like curve over $\mathbb{C}$, and $E$ be a vector bundle of rank $n$ on $C$. In this paper, we investigate the criteria for the semistability of the restriction of $E$ onto the components of $C$ 
when $E$ is given to be semistable with respect to a polarization $w$. 
As an application, assuming each irreducible component of $C$ is general in its moduli space, we investigate the $w$-semistability of kernel bundles on such curves, extending the results (completely for rank two and partially for higher rank) known in the case of a reducible nodal curve with two smooth components, but here, using different techniques.
\end{abstract}

\section{Introduction}
Let $C$ be a curve of compact type having $N\geq 2$ components $C_1,\dots,C_N$ over the field of complex numbers. Let $E$ be a vector bundle of rank $n$ on $C$ obtained by gluing vector bundles $E_j$ of rank $n$ on $C_j$, and let $\chi$ and $\chi_j$ denote the Euler characteristics of $E$ and $E_j$ respectively. Let $w=(w_1,\dots,w_N)$ be a fixed $n$-tuple of positive rational numbers such that $\sum\limits_{j=1}^N w_j =1$ (called polarization). If $\chi_j$ and $w_j$ satisfy the inequalities 
\begin{equation}\label{polarization inequalities}
   (\sum_{i=1}^j w_i)\chi - \sum_{i=1}^{j-1}\chi_i + n(j-1) \leq \chi_j \leq (\sum_{i=1}^i w_i)\chi - \sum_{i=1}^{j-1}\chi_i + nj,
     \end{equation}
for $j = 1,\ldots N-1$ and $E_j$ is a semistable vector bundle for each $j$, then by \cite[Theorem 1, Step II]{Bigas91}, the bundle $E$ is semistable with respect to the polarization $w = (w_1,\dots,w_N)$. Conversely, if $E$ is semistable with respect to $w$, then by \cite[Theorem 1, Step I]{Bigas91}, it is known that $\chi$, $\chi_j$ and $w_j$ satisfy the inequalities \eqref{polarization inequalities}.However, if we start with a $w$-semistable bundle $E$, a natural question arises: can we conclude that each $E_j$ is semistable? More precisely, if $E_j$ is a vector bundle of rank $n$ on $C_j$ for each $j$, and $E$, which is a bundle of rank $n$ on $C$ obtained by gluing $E_j$, is known to be $w$-semistable, does this imply that each $E_j$ is semistable?

     If $C$ is reducible nodal curve with two smooth components $C_1$ and $C_2$ intersecting at a node $p$, and $E$ is a vector bundle on $C$ of rank two such that $E$ is semistable with respect to a fixed polarization $w = (w_1,w_2)$, 
     a theorem of Nagaraj-Seshadri (\cite[Theorem 5.1]{nagaraj-seshadri96}) proves that both $E_1$ and $E_2$ are semistable, provided $w_1 \chi \notin \mathbb{Z}$.      
     
In this article, we address this question by focusing on a particular class of reducible nodal curves of compact type, referred to as a \textit{comb-like curves}. 
In section \ref{section:2}, we begin by recalling the definition of these curves and observe that on such curves, the inequalities \eqref{polarization inequalities} simplify to 
     \[ w_j \chi \leq \chi_j \leq w_j \chi + n, \]
for $j = 1, \dots , N-1$ (see Theorem \ref{Ineq:Teix-Bigas}). Subsequently, we prove the following: 
    
     \begin{reptheorem}{rank 2 necessary condition}
Let $C$ be a comb-like curve and $E$ be a $w$-semistable vector bundle of rank 2 on $C$. Then for $j \in \{1,\ldots, N-1\}$, if $w_j \chi \notin \mathbb{Z}$, the following holds: 
\begin{enumerate}
    \item[(i)] If $w_j \chi + 1 < \chi_j < w_j \chi + 2$, then $E_j$ is semistable.
    \item[(ii)] If $\chi_j$ is even, then $E_j$ is semistable.
    \item[(iii)] If $E_j$ is not semistable, then the destabilizing line subbundle $L_j$ of $E_j$ has to satisfy $\chi(L_j) = \frac{\chi_j}{2} + \frac{1}{2}$.
\end{enumerate}
     \end{reptheorem}
 
For the higher rank case, we prove a similar result but with stronger assumptions:

\begin{reptheorem}{theorem for vb of rank n}
Let $C$ be a comb-like curve and $E$ be a $w$-semistable vector bundle of rank $n \geq 2$ on $C$. Then for $j \in \{1,\ldots, N-1\}$, if $w_j \chi \notin \mathbb{Z}$, the following holds:
\begin{enumerate}
    \item[(i)] If $n$ divides $\chi_j$ and $w_j \chi + (n-1) < \chi_j < w_j \chi + n$, then $E_j$ is semistable.
    \item[(ii)] If $n$ divides $\chi_j$, $E_j$ is not semistable and $L_j \subset E_j$ is a destabilizing subbundle of rank $k$, then $\chi(L_j) = \frac{k\chi_j}{n} + a$ for some $a \in \{1,\dots,k-1\}$. 
    \item[(iii)] If $E_j$ is not semistable and $L_j \subset E_j$ is a destabilizing subbundle of rank $k$ such that $k$ divides $\chi(L_j)$, then $\frac{\chi(L_j)}{k} = \frac{\chi_j}{n} + \frac{(n-r_j)}{n}$, where $r_j$ is the remainder obtained when $\chi_j$ is divided by $n$.
\end{enumerate}
\end{reptheorem}
A crucial point that comes out of these results is that under the above mentioned hypotheses, even if $E_j$ is not possibly semistable, we can precisely determine all possible values for the Euler characteristics of the destabilizing subbundles of $E_j$ in terms of the ranks of those subbundles and the Euler characteristic of $E_j$. This plays an important role in section \ref{Application to Kernel Bundles}. 

As an application of these results, in section \ref{Application to Kernel Bundles} we investigate the $w$-semistability of \textit{kernel bundle} associated to a \textit{generated pair} on $C$. Let $E$ be a globally generated vector bundle on $C$ of rank $n$, and $V \subseteq H^0(C,E)$ be a subspace of dimension $l > n$ that generates $E$. We call such a pair $(E,V)$, a generated pair on $C$. The kernel bundle $M_{E,V}$ associated to the generated pair $(E,V)$ is then defined by the following exact sequence:
 \begin{equation*}\label{defining sequence for the kernel bundle}
0 \rightarrow M_{E,V} \rightarrow V \otimes \mathcal{O}_C \rightarrow E \rightarrow 0.
\end{equation*}
When $V = H^0(C,E)$, we denote the kernel bundle by $M_E$. The semistablity of $M_{E,V}$ over a smooth projective curve has been a longstanding problem. In \cite{Butler-conj},  Butler proposed a conjecture regarding the semistability of $M_{E, V}$ for a semistable vector bundle $E$ on $C$, one version of which states:
 \begin{conjecture}\label{Butler-conjecture}
Let $C$ be a general smooth curve of genus $g \geqslant 3$. Then, for a general choice of a generated pair $(E, V)$ on $C$, where $E$ is a semistable vector bundle, the kernel bundle $M_{E, V}$ is semistable.
\end{conjecture}
Even before the conjecture was made, some specific cases were proven. For instance, over a smooth projective curve of genus $g \geq 2$, Ein and Lazarsfeld demonstrated that $M_{L}$ is (semi)stable whenever $L$ is a line bundle of $\deg L  (\geqslant) > 2g$ \cite[Proposition 3.1]{Ein92}. Butler generalized this result for the case of a (semi)stable vector bundle of any rank \cite[Theorem 1.2]{Butler94}. One of the motivations for studying this question of semistability of $M_{E,V}$ was because of its deep connections with higher rank Brill-Noether theory \cite{Beauville2006}.

In 2015, Bhosle, Brambila-Paz and Newstead proved the conjecture for the case of a line bundle $L$ \cite{Brambila-Newstead2019}. However, for the higher rank case, it still remains open, although significant work has been done to address this conjecture (see for example,\cite{Brambila-Newstead2019}). 

More recently, Brivio and Favale studied the (semi)stability of $M_{E, V}$ over a reducible nodal curve with two smooth components intersecting at a node \cite{Brivio-Favale2020}. Subsequently, Suhas B N, Majumdar and Singh investigated the (semi)stability of $M_{E, V}$ over chain-like curves \cite{SSA}. Surprisingly, these results on reducible nodal curves show that $M_{E,V}$ is almost always strongly unstable (see Definition \ref{defn: w-semistable and strongly unstable}). They also provided a sufficient condition under which $M_{E,V}$ becomes semistable with respect to a polarization.

We extend the results of Brivio-Favale for the case of a comb-like curve $C$. Notably, the techniques employed here for proving the results on strong unstability of $M_{E,V}$ differs from theirs and use much lesser hypothesis. To be more precise, in \cite[Theorem 2.4]{Brivio-Favale2020}, it is assumed as a part of the hypothesis that $\ker ({\rho_j}_{\mid_V}) \neq 0$ and $E_j$ is semistable for each $j$. However, we require $\ker ({\rho_j}_{\mid_V}) \neq 0$ for some $j \in \{1,\dots,N-1\}$ and $\text{deg}(E_j)$ to avoid a specific value for that particular $j$ (if rank($M_{E,V}) > 2$). We do not need any assumption on the semistability of $E_j$. More precisely, we prove the following result:
\begin{reptheorem}{cor:  on strongly unstable}
Let $(E, V)$ be a generated pair on a comb-like curve $C$. Then the following holds:
\begin{enumerate}
    \item[(i)] If $\text{rank}(M_{E,V}) = 2$, and $\ker ({\rho_j}_{\mid_V}) \neq 0$ for some $j$, then $M_{E,V}$ is strongly unstable.
    \item[(ii)] If $\text{rank}(M_{E,V}) = (l-n) > 2$, $\ker ({\rho_j}_{\mid_V}) \neq 0$ for some $j \in \{1,\dots,N-1\}$ and $d_j \neq (l-n) - r_j$ for that particular $j$ (where $r_j$ is the reminder obtained when $\chi_j$ is divided by (l-n)), then $M_{E, V}$ is strongly unstable.
\end{enumerate}
\end{reptheorem}
We then provide a sufficient condition for $M_{E,V}$ to be semistable with respect to a polarization $w$ (Theorem \ref{thm: existence of polarization}). At the end, we assume that the components of $C$ are general members of their corresponding moduli spaces. We then give a complete characterization for $M_{L,V}$ to be $w$-semistable (Theorem \ref{thm: about (L,V)}), where $L$ is a line bundle on $C$ such that $\text{deg}(L_j)$ satisfies condition as in item (ii) of Theorem \ref{cor:  on strongly unstable}. Finally, assuming Butler's conjecture, we provide a complete characterization for $M_{E,V}$ to be $w$-semistable (Theorem \ref{The iff condition}) with the same condition on $\text{deg}(E_j)$ as in item (ii) of Theorem \ref{cor:  on strongly unstable}.

\section{Semistability of the restrictions}\label{section:2}
We begin by recalling the definition of a comb-like curve \cite[Example 2.3]{Brivio-Favale2023}. Let $N\geq 2$ be a positive integer and let $C$ be a complex reduced projective curve having $N$ smooth components $C_j$ of genus $g_j$ and $N-1$ nodes $p_j$ such that $C_i \cap C_j = \emptyset$ for $j \in \{1, \dots, N-1\}$, $i \notin \{j, N\}$ and $C_{j} \cap C_N = \{ p_j \}$, for $j \in \{1, \dots, N-1\}$. We call such a curve a comb-like curve with the fixed component $C_{N}$. Throughout this section, by $C$, we always mean a curve of the type mentioned here.

\noindent On such a curve, we have the following exact sequence:

 \begin{equation}\label{canonical exact sequence}
    0 \rightarrow \mathcal{O}_C \rightarrow \bigoplus_{j=1}^{N
    }\mathcal{O}_{C_j} \rightarrow \mathcal{T} \rightarrow 0,
 \end{equation}
 where $\mathcal{T}$ is supported only at the nodal point(s).
 Also, it's not hard to show that each $\mathcal{T}_{p_j}$ is a one dimensional vector space over $\mathbb{C}$ and that $h^0(\mathcal{T}) = (N-1)$.
 From the exact sequence \ref{canonical exact sequence}, we have 
\begin{eqnarray*}
    \chi(\mathcal{O}_C) & = & \sum_{j=1}^{N} \chi(\mathcal{O}_{C_j}) - \chi(\mathcal{T}) \nonumber \\
    & = & N - \sum_{j=1}^{N} g_j - (N-1)
    \nonumber \\
    & = & 1 - \sum_{j=1}^N g_j
\end{eqnarray*}
So if we let $p_a(C) = 1 - \chi(\mathcal{O}_C)$ be the arithmetic genus of $C$, then $p_a(C) = \sum\limits_{j=1}^{N} g_j$.

Now suppose $E$ is a vector bundle of rank $n$ on $C$, and $E_{|_{C_j}}$ is denoted by $E_j$. Then we have the short exact sequence
\begin{equation}\label{canonical sequence for the vector bundle E}
  0 \rightarrow E \rightarrow \bigoplus_{j=1}^{N} E_j \rightarrow \mathcal{T}_E \rightarrow 0,
\end{equation}
where $\mathcal{T}_E$ is a torsion sheaf supported only at the nodal points. We also have 
\begin{eqnarray}\label{euler char of vb of C}
\chi(E) = \sum_{j=1}^{N} \chi(E_j) - n(N-1).
\end{eqnarray}

\begin{definition}
 Let $F$ be a coherent sheaf of $\mathcal{O}_C$-modules. We call $F$ a pure sheaf of dimension one if for every proper $\mathcal{O}_C$-submodule  $G \subset F$ and $G \neq 0$, the dimension of support of $G$ is equal to one.
\end{definition}
Vector bundles on $C$ are examples of pure sheaves of dimension one. Suppose $F$ is a pure sheaf of dimension one on $C$. Let $F_j = \frac{F_{|_{C_j}}}{\text{Torsion}(F_{|_{C_j}})}$ for each $j$, where $\text{Torsion}(F_{|_{C_j}})$ is the torsion subsheaf of $F_{|_{C_j}}$. Then $F_j$, if non-zero, is torsion-free and hence locally free on $C_j$. Let $r_j$ denote the rank of $F_j$. Also let $d_j$ denote the degree of $F_j$ for each $j$. Then the $n$-tuples $(r_1,\dots,r_n)$ and $(d_1,\dots,d_n)$ respectively, are called the \textit{multirank} and \textit{multidegree} of $F$ (\cite[page 5]{Brivio-Favale2020}).

We now recall the definition of a polarization and semistability with respect to a polarization (see \cite{S} for more details).
\begin{definition}\label{polarization}
 Let $w = (w_1,\dots,w_n)$ be an $n$-tuple of rational numbers such that $0 < w_j < 1$ for each $j$ and $\sum\limits_{j=1}^n w_j =1$. Then such an $n$-tuple is called a polarization on $C$.
\end{definition}

\begin{definition}\label{polarized slope}
 Suppose $F$ is a pure sheaf of dimension one on $C$ of multirank $(r_1,\dots, r_n)$. Then the slope of $F$ with respect to a polarization $w$, denoted by $\mu_w(F)$, is defined by $\mu_w(F) = \frac{\chi(F)}{\sum\limits_{j=1}^n w_jr_j}$.
\end{definition}

\begin{definition} \label{defn: w-semistable and strongly unstable}
 Let $E$ be a vector bundle defined on $C$. Then $E$ is said to be $w$-semistable (resp. $w$-stable) if for any proper subsheaf $F \subset E$ one has $\mu_w(F) \leq \mu_w(E)$ (resp. $\mu_w(F) < \mu_w(E)$). If $E$ is not $w$-semistable for any polarization $w$ on $C$, then such an $E$ is called a strongly unstable bundle (\cite[Definition 1.4]{Brivio-Favale2020}).
\end{definition}

The following result appears in \cite[Theorem 1, Step-1]{Bigas91} in a more general form. For the case of a comb-like curve, the result is much simpler. So, we state and prove it here.

\begin{theorem}\label{Ineq:Teix-Bigas}
Let $E$ be a vector bundle of rank $n$ on $C$ which is semistable with respect to a polarization $w = (w_1, \dots , w_N)$. If $\chi_j$ and $\chi$ denote the Euler characteristics of $E_j$ and $E$ respectively, then we have 
\[ w_j \chi \leq \chi_j \leq w_j \chi + n, \]
for $j = 1, \dots , N-1$.
\begin{proof}
We know that $E_j (-p_j) \hookrightarrow E$, for $j = 1, \dots , N-1$. Then by using $w$-semistability of $E$,
\[
\begin{split}
\frac{\chi(E_j (-p_j))}{n w_j}  \leq \frac{\chi(E)}{n} & \Rightarrow \frac{\chi(E_j) -n}{n  w_j} \leq \frac{\chi(E)}{n}.
\end{split}
\]
Therefore,
\begin{eqnarray}\label{inq:1}
 \chi_j \leq w_j \chi + n.
\end{eqnarray}
\noindent
Now, let $\widetilde{E}_j$ denote the kernel of the map \[ \bigoplus_{\substack{i = 1 \\ i \neq j}}^{N-1} E_i \oplus E_N(-p_j) \rightarrow \tau_{\widetilde{E}_j} \rightarrow 0,
\] where $\tau_{\widetilde{E}_j}$ is the corresponding torsion sheaf. Then 
\[
\begin{split}
& \; \; \frac{\chi(\widetilde{E}_j)}{(\sum\limits_{\substack{i = 1 \\ i \neq j}}^N w_i)n}   \leq  \frac{\chi}{n}  \\
 \Rightarrow &  \; \; \chi(\widetilde{E}_j) \leq (1 - w_j) \chi \\
 \Rightarrow & \; \; \sum\limits_{\substack{i = 1 \\ i \neq j}}^N \chi_i - n (N-1) \leq (1 -w_j) \chi \\
 \Rightarrow & \; \;  \chi - \chi_j \leq \chi - w_j \chi 
\end{split}
\]
Therefore, we obtain 
\begin{eqnarray}\label{inq:2}
 \chi_j \geq w_j \chi.
\end{eqnarray}
Combining \eqref{inq:1} and \eqref{inq:2}, we have the desired result.
\end{proof}
\end{theorem}
We now state and prove the main results of this section.

\begin{theorem} \label{rank 2 necessary condition}
Let $E$ is a $w$-semistable vector bundle of rank 2 on $C$. Then for $j \in \{1,\ldots, N-1\}$, if $w_j \chi \notin \mathbb{Z}$, the following holds: 
\begin{enumerate}
    \item[(i)] If $w_j \chi + 1 < \chi_j < w_j \chi + 2$, then $E_j$ is semistable.
    \item[(ii)] If $\chi_j$ is even, then $E_j$ is semistable.
    \item[(iii)] If $E_j$ is not semistable, then the destabilizing line subbundle $L_j$ of $E_j$ has to satisfy $\chi(L_j) = \frac{\chi_j}{2} + \frac{1}{2}$.
\end{enumerate}
\begin{proof}
By the previous theorem, for $j = 1,\ldots, N-1$, we have 
\[
w_j \chi < \chi_j < w_j \chi + 2.
\] 

\noindent
(i) Suppose that $w_j \chi + 1 < \chi_j < w_j \chi + 2$, for some $ j \in \{ 1, \dots ,  N-1 \}$. Let $L_j \hookrightarrow E_j$ be a line subbundle. This implies that $L_j(-p_j)$ is a subsheaf of $E$, and so, by the $w$-semistability of $E$, we have
\[
\frac{\chi(L_j) - 1}{w_j} \leq \frac{\chi}{2}. 
\]
Then
\[
\begin{split}
\chi(L_j) & \leq \frac{w_j \chi}{2} +1 \\
          & < \frac{\chi_j -1}{2} + 1~~~~~~\text{(by the choice of $\chi_j$)} \\
          & = \frac{\chi_j}{2} + \frac{1}{2}.
\end{split}
\]
This implies that $\chi(L_j) \leq \frac{\chi_j}{2}$. Therefore, $E_j$ is semistable in this case. (If $\chi_j$ is even, then $E_j$ is stable as well, for, in that case, we have $\chi(L_j) < \frac{\chi_j}{2}$). \\

\noindent (ii) Now suppose that $\chi_j$ is even for some $j \in \{1,\dots,N-1\}$. If $\chi_j$ satisfies the inequalities $w_j \chi + 1 < \chi_j < w_j \chi + 2$, then by item (i), we are done with the proof. So, we suppose that $w_j \chi < \chi_j < w_j \chi + 1$. Then by the similar argument as above, we have 
\[ \chi(L_j) < \frac{\chi_j}{2} + 1,\]
and so, 
 \begin{equation} \label{subbundle violating semistability}
     \chi(L_j) \leq \frac{\chi_j}{2} + \frac{1}{2}. 
 \end{equation}
 
 As $\chi_j$ is even, we have $\chi(L_j) \leq \frac{\chi_j}{2}$, and so, $E_j$  is semistable. \\
 
 \noindent (iii) Suppose that $E_j$ is not semistable. This implies from items (i) and (ii) that $\chi_j$ is odd and that it satisfies $w_j \chi < \chi_j < w_j \chi + 1$. So, if $L_j$ is any subbundle of $E_j$, then by inequality \eqref{subbundle violating semistability}, we have $\chi(L_j) \leq \frac{\chi_j}{2} + \frac{1}{2}$. Now, as $E_j$ is not semistable, this means there must exist a destabilizing subbundle $L_j$ of $E_j$. From the above arguments, this implies that $\chi(L_j)$ has to be 
 equal to $\frac{\chi_j}{2} + \frac{1}{2}$.

\end{proof}
\end{theorem}

More generally, Let $E$ be a vector bundle of rank $n > 2$ on $C$. Under some extra assumptions on the Euler characteristics, we get a similar result as in the case of rank $2$.

\begin{theorem}\label{theorem for vb of rank n}
Let $E$ be a $w$-semistable vector bundle of rank $n$ on $C$. 
Then for $j \in \{1,\ldots, N-1\}$, if $w_j \chi \notin \mathbb{Z}$, the following holds:
\begin{enumerate}
    \item[(i)] If $n$ divides $\chi_j$ and $w_j \chi + (n-1) < \chi_j < w_j \chi + n$, then $E_j$ is semistable.
    \item[(ii)] If $n$ divides $\chi_j$, $E_j$ is not semistable and $L_j \subset E_j$ is a destabilizing subbundle of rank $k$, then $\chi(L_j) = \frac{k\chi_j}{n} + a$ for some $a \in \{1,\dots,k-1\}$. 
    \item[(iii)] If $E_j$ is not semistable and $L_j \subset E_j$ is a destabilizing subbundle of rank $k$ such that $k$ divides $\chi(L_j)$, then $\frac{\chi(L_j)}{k} = \frac{\chi_j}{n} + \frac{(n-r_j)}{n}$, where $r_j$ is the remainder obtained when $\chi_j$ is divided by $n$.
\end{enumerate}

\end{theorem}
\begin{proof}
(i) Suppose that $n$ divides $\chi_j$ and $\chi_j$ satisfies 
\begin{eqnarray}\label{Rankn:inequality:1}
w_j \chi + (n-1) < \chi_j < w_j \chi + n, 
\end{eqnarray}
\text{for some} $j \in \{1, \dots , N-1\}$.
Let $L_j$ be a subbundle of $E_j$ of rank $k$. Then $L_j(-p_j)$ is a subsheaf of $E$. So, by $w$-semistability of $E$, we have 

\begin{eqnarray}\nonumber
    \frac{\chi(L_j) - k}{k\; w_j} &\leq& \frac{\chi}{n} \\ \label{Rankn:inequality:2}
    \Rightarrow \frac{\chi(L_j)}{k} &\leq& \frac{w_j\chi}{n} +1 \\ \nonumber
    &<& \frac{\chi_j}{n} + 1 -(\frac{n-1}{n}) \quad \quad \text{(from \eqref{Rankn:inequality:1})} \\ \nonumber
    \Rightarrow \chi(L_j) &<& \frac{k \chi_j}{n} + \frac{k}{n}.
\end{eqnarray}
As $\frac{k}{n} < 1$ and $n\mid \chi_j$, we have  $\frac{\chi(L_j)}{k} \le \frac{\chi_j}{n}$, thereby proving the semistability of $E_j$. \\

\noindent (ii) Again, suppose that $n$ divides $\chi_j$. If $E_j$ is not semistable, then by item (i), $\chi_j$ satisfies
\begin{eqnarray}\label{Rankn:inequality:3}
    w_j \chi  < \chi_j < w_j \chi + n-2.
\end{eqnarray}
As before, if $L_j$ is a subbundle of $E_j$ of rank $k$, we have 
\[
\chi(L_j) < \frac{k\chi_j}{n} + k, \quad \quad \text{(using \eqref{Rankn:inequality:2}, and \eqref{Rankn:inequality:3})}.
\]
Therefore the only way in which $L_j$ becomes a destabilizing subbundle for $E_j$ is when  
 \[
 \chi(L_j) = \frac{k \chi_j}{n} + a
 \]
 where $1 \le a \le k-1$. This completes the proof of (ii).\\

 \noindent (iii) Suppose that $E_j$ is not semistable and $L_j \subset E_j$ is a destabilizing subbundle of rank $k$ such that $k$ divides $\chi(L_j)$. Now, if $\chi_j = nq_j + r_j$ where $0\leq r_j \leq n-1$, then $\frac{\chi_j}{n} = q_j + \frac{r_j}{n}$. This means that if $L_j$ is a destabilizing subbundle of $E_j$, the only possibility for $\frac{\chi(L_j)}{k}$ (which is an integer), is $\frac{\chi_j}{n} + \frac{n - r_j}{n}$.

\end{proof}
Note that the proof of Theorem \ref{theorem for vb of rank n} gives more information than just the statement. More precisely, we have the following:
\begin{proposition}\label{line subbundle cant destabilize}
 Let $E$ be a w-semistable vector bundle of rank $n$ on $C$. Suppose that for some $j \in \{1,\dots,N-1\}$, $w_j\chi \notin \mathbb{Z}$, $n$ divides $\chi_j$ and $L_j \subset E_j$ is subbundle of rank $k$ such that $k$ divides $\chi(L_j)$. Then $L_j$ can't destabilize $E_j$. In particular, a line subbundle of $E_j$ can't destabilize $E_j$.
\end{proposition}
\begin{proof}
    By the arguments of Theorem \ref{theorem for vb of rank n}, we have $\frac{\chi(L_j)}{k} < \frac{\chi_j}{n} + 1$. So, if $n$ divides $\chi_j$ and $k$ divides $\chi(L_j)$, then the above inequality implies that $\frac{\chi(L_j)}{k} \leq \frac{\chi_j}{n}$. In particular, this means that such a subbundle $L_j$ cannot destabilize $E_j$.
\end{proof}

\section{Application to kernel bundles} \label{Application to Kernel Bundles}
Let $N$ be a positive integer and $C$ be a comb-like curve with $N$ components $C_1,\dots,C_N$, where each $C_j$ has genus $g_j \geq 2$. Suppose that $E$ is a vector bundle of rank $n$ and $V \subseteq H^0(C,E)$ is a subspace of dimension $l > n$ such that $V$ generates $E$. Such a pair $(E, V)$ is called a generated pair on $C$. The kernel bundle $M_{E, V}$ on $C$ of the generated pair $(E, V)$ is defined by the following exact sequence:
\begin{eqnarray}\label{eq; definig kernel bundle}
0 \rightarrow M_{E,V} \rightarrow V \otimes \mathcal{O}_C \rightarrow E \rightarrow 0.
\end{eqnarray}
These notions are quite classical. When the base curve is smooth, they have been studied, for example, in \cite{paranjape-ramanan, Ein92}, and when the base curve is a certain type of reducible nodal curve, they appear in \cite{SSA, Brivio-Favale2020}.

We have the following exact sequence on $C$ for each $j \in \{ 1, \dots , N \}$:
\begin{eqnarray}\label{Exact sequence defining I_C_i}
0 \rightarrow \mathcal{I}_{C_j} \rightarrow \mathcal{O}_C \rightarrow \mathcal{O}_{C_j} \rightarrow 0,
\end{eqnarray}
where $\mathcal{I}_{C_j}$ is an ideal sheaf of $C_j$ in $C$. It is a sheaf such that 
\begin{center}
    $\mathcal{I}_{C_j} \supseteq  (\bigoplus\limits_{\substack{i=1 \\ i \neq j}}^{N-1} \mathcal{O}_{C_i}(-p_i)) \oplus \mathcal{O}_{C_N}(-p_1 - \cdots - p_{N-1}) \; \text{and}$ \\
    $ \mathcal{I}_{C_N} = \bigoplus\limits_{i=1}^{N-1} \mathcal{O}_{C_i}(-p_i)$.
\end{center}

If $(E, V)$ is a generated pair on $C$, then for each $j$, we have the natural morphisms $\rho_j : H^0(C, E) \rightarrow H^0(C_j, E_j)$ defined by $\rho_j(s) = s_{\mid_{C_j}}$. We denote $\rho_j(V)$ by $V_j$ for each $j$. Tensoring \eqref{Exact sequence defining I_C_i} by $E$, we have 
\begin{center} \label{Exact sequence defining I_C_i tensor E}
 $0 \rightarrow E \otimes \mathcal{I}_{C_j} \rightarrow E \rightarrow E_j \rightarrow 0.$
\end{center}
Now, passing on to the long exact sequence in the cohomology, we obtain
\begin{center}
$0 \rightarrow H^0(E \otimes \mathcal{I}_{C_j}) \rightarrow H^0(E) \xrightarrow{\rho_j} H^0(E_j) \rightarrow \cdots$
\end{center}
This implies, in particular, 
\begin{center}\label{ker:N}
$\text{ker}(\rho_{{N}_{|_V}}) = \bigoplus\limits_{i=1}^{N-1} (V \cap H^0(E_i(-p_i))),$
\end{center}
and for $j \in \{1,\dots, N-1 \}$, 
\begin{center}\label{ker:j<N}
$\text{ker}(\rho_{j_{|_V}}) \supseteq  \big( \bigoplus\limits_{\substack{i = 1 \\ i \neq j}}^{N-1} (V \cap H^0(E_i(-p_i))) \big) \oplus (V \cap H^0(E_N(-p_1-\cdots -p_{N-1}))).$
\end{center}

It is not hard to see that if $(E, V)$ is a generated pair on $C$, then $(E_j, V_j)$ is a generated pair on $C_j$ (proof of this is similar to \cite[Lemma 3.1]{SSA}). Further, we have the following commutative diagram:
\[
\begin{tikzcd}
  & 0 \arrow[d, ]  & 0 \arrow[d, ]  & 0 \arrow[d, ]  \\
   0 \arrow[r] & {Ker(\rho_{j_{|_V}})} \otimes \mathcal{O}_{C_j} \arrow[equal, ]{d} \arrow[r, ] & M_{E,V}\otimes \mathcal{O}_{C_j} \arrow[d, ] \arrow[r] & M_{E_j, V_j} \arrow[d, ] \arrow[r] & 0 \\
    0 \arrow[r] & {Ker(\rho_{j_{|_V}})} \otimes \mathcal{O}_{C_j} \arrow[d, ] \arrow[r, ] & V\otimes \mathcal{O}_{C_j} \arrow[d, ] \arrow[r,] & V_j \otimes \mathcal{O}_{C_j} \arrow[d, ] \arrow[r] & 0 \\
     & 0  \arrow[r, ] & E_j \arrow[d, ] \arrow[equal]{r} & E_j \arrow[d, ] \arrow[r] & 0\\
  & &0  & 0   
\end{tikzcd}
\] 

From \eqref{eq; definig kernel bundle}, it is clear that $\text{rk}(M_{E, V}) = l -n > 0$ and $\deg(M_{E,V}) = -d$, where $d = \sum\limits_{j =1}^N\deg(E_j)$. Similarly, if we let $d_j = \deg(E_j)$, then $\deg (M_{E_j, V_j}) = \deg(M_{E,V} \otimes \mathcal{O}_{C_j}) = -d_j$. We shall denote $\chi(M_{E,V})$ by $\chi$ and $ \chi(M_{E,V}) \otimes \mathcal{O}_{C_j}$ by $\chi_j$  throughout this section.\\

\begin{proposition}\label{prop: on unstable bundle}
Let $(E, V)$ be a generated pair on $C$. If $\ker({\rho_j}_{\mid_V}) \neq 0$ for some $j$, then $M_{E, V} \otimes \mathcal{O}_{C_j}$ is unstable.

\begin{proof}
As $\deg(\ker ({\rho_j}_{\mid_V}) \otimes \mathcal{O}_{C_j}) = 0$, $\mu(\ker ({\rho_j}_{\mid_V}) \otimes \mathcal{O}_{C_j}) = 0$. But 
\[
\mu(M_{E, V} \otimes \mathcal{O}_{C_j}) = \frac{-d_j}{l-n} <0.
\]
This implies that $\mu(\ker ({\rho_j}_{\mid_V}) \otimes \mathcal{O}_{C_j}) > \mu(M_{E, V} \otimes \mathcal{O}_{C_j})$. Therefore, $M_{E, V} \otimes \mathcal{O}_{C_j}$ is unstable if $\ker({\rho_j}_{\mid_V}) \neq 0$.
\end{proof}
\end{proposition}
\begin{remark} \label{crucial remark after Proposition 3.1}
   Note that if $\ker ({\rho_j}_{\mid_V}) \neq 0$ for some $j \in \{1,\dots,N-1\}$, then by Proposition \ref{line subbundle cant destabilize}, $(l-n)$ cannot divide $\chi_j$ for that particular $j$.  
\end{remark}
We need the following result for our next theorem. 
\begin{theorem}[{\cite[Theorem 1]{Sierra09}}] \label{degree bound for gg vb}
    Let $E$ be a globally generated vector bundle of rank $e \geq 2$ on a smooth projective variety $X$ and let $L := \det (E)$. If $\deg (E) > 0$ and $E$ is not isomorphic to $\mathcal{O}_X^{e-1} \oplus L$, then 
    \[
    \deg (E) \geq  h^0(X, E) - e.
    \]
Moreover, if equality holds then $\deg(L) = h^0(X, L)- n$ and, in particular, $X$ is rational.
\end{theorem}
Note that the above mentioned bound on $\deg (E)$ holds even when $E$ is isomorphic to $\mathcal{O}_X^{e-1} \oplus L$, provided that $X$ is a curve. This is evident from the discussion preceding Theorem 1 of \cite{Sierra09}.
Now, as an application of Proposition \ref{prop: on unstable bundle}, Theorems \ref{rank 2 necessary condition} and \ref{theorem for vb of rank n}, we have the following result:
\begin{theorem}\label{cor:  on strongly unstable}
Let $(E, V)$ be a generated pair on $C$. Then the following holds:
\begin{enumerate}
    \item[(i)] If $\text{rank}(M_{E,V}) = 2$, and $\ker ({\rho_j}_{\mid_V}) \neq 0$ for some $j$, then $M_{E,V}$ is strongly unstable.
    \item[(ii)] If $\text{rank}(M_{E,V}) = (l-n) > 2$, $\ker ({\rho_j}_{\mid_V}) \neq 0$ for some $j \in \{1,\dots,N-1\}$ and $d_j \neq (l-n) - r_j$ for that particular $j$ (where $r_j$ is the reminder obtained when $\chi_j$ is divided by (l-n) and $r_j > 0$), then $M_{E, V}$ is strongly unstable.
\end{enumerate}

\begin{proof} 
Suppose that $M_{E,V}$ is semistable with respect to a polarization $w$. We consider each case separately and arrive at a contradiction. First, observe that if $\ker ({\rho_N}_{\mid_V}) \neq 0$, it will imply that $V \cap H^0(E_i(-p_i)) \neq 0$ for some $i \in \{1,\dots,N-1\}$, which in turn implies that $\ker ({\rho_j}_{\mid_V}) \neq 0$ for all $j \neq i$. So, without loss of generality, we assume that $\ker ({\rho_j}_{\mid_V}) \neq 0$ for some $j \in \{1,\dots, N-1\}$ in item (i) as well. \\

(i) As $\ker ({\rho_j}_{\mid_V}) \neq 0$, by Proposition \ref{prop: on unstable bundle}, the bundle $\ker ({\rho_j}_{\mid_V}) \otimes \mathcal{O}_{C_j}$ becomes a destabilizing subbundle for $M_{E, V} \otimes \mathcal{O}_{C_j}$. However, by Theorem \ref{rank 2 necessary condition}(iii), the only way in which $M_{E, V} \otimes \mathcal{O}_{C_j}$ violates semistability is when there exists a line subbundle $L_j$ of $M_{E, V} \otimes \mathcal{O}_{C_j}$ such that $\chi(L_j) = \frac{\chi_j}{2} + \frac{1}{2}$. This means that 
    \begin{eqnarray*}
        \chi(\ker ({\rho_j}_{\mid_V}) \otimes \mathcal{O}_{C_j}) & = & \frac{\chi_j}{2} + \frac{1}{2} \\
       \Rightarrow (1-g_j) & = & \frac{-d_j + 2 (1-g_j)}{2} + \frac{1}{2}, 
    \end{eqnarray*}
    which forces $d_j =1$. We claim that if $d_j =1$, then $E_j$ can't be globally generated. This contradicts our choice of $E_j$, thereby completing the proof. 

    To prove the claim, suppose that $d_j = 1$. If $E_j$ is globally generated, then by  Theorem \ref{degree bound for gg vb}, we have $d_j \geq h^0(E_j) - n$. This implies $h^0(E_j) \leq n+1$. As $E_j$ is globally generated, this further implies that $h^0(E_j) = n~\text{or}~n+1$. If it is $n$, it forces $E_j$ to be isomorphic
    to trivial sheaf of rank $n$, which is a contradiction as their degrees won't match. If $h^0(E_j) =n+1$, then $1=d_j=h^0(E_j) - n$, which, again by Theorem \ref{degree bound for gg vb}, implies $C_j$ is rational. But $C_j$ has genus $g_j \geq 2$. So, this also is not possible. Therefore, $E_j$ can not be globally generated if it has degree equal to $1$. \\

(ii) Again as $\ker ({\rho_j}_{\mid_V}) \neq 0$, if $k$ denotes $\text{dim}(\ker ({\rho_j}_{\mid_V}))$, by Theorem \ref{theorem for vb of rank n}(iii), we have
     \begin{eqnarray*}
        \frac{\chi(\ker ({\rho_j}_{\mid_V}) \otimes \mathcal{O}_{C_j})}{k} & = &  \frac{\chi_j}{(l-n)} + \frac{(l-n-r_j)}{(l-n)} \\
        \Rightarrow (1-g_j) & = & \frac{-d_j + (l-n)(1-g_j)}{l-n} + \frac{(l-n-r_j)}{(l-n)},
    \end{eqnarray*}
    which forces $d_j = l-n-r_j$. This is a contradiction to the given hypothesis. 
\end{proof}
\end{theorem}

From Theorem \ref{cor:  on strongly unstable}, it is clear that for $M_{E,V}$ to be semistable with respect to a polarization $w$, it is reasonable to start with the assumption that $\ker ({\rho_j}_{\mid_V}) =0$ for each $j$. The following theorem says that we need a little more than that to ascertain the existence of such a polarization.

\begin{theorem}\label{thm: existence of polarization}
Let $(E,V)$ be a generated pair on $C$ and $M_{E,V}$ the associated kernel bundle. If $M_{E, V} \otimes \mathcal{O}_{C_j}$ is semistable for each $j$, 
then there exists a polarization $w = (w_1, \dots , w_N)$ such that $M_{E, V}$ is $w$-semistable.

\begin{proof}
Suppose that there exists a polarization $w= (w_1, \dots , w_N)$ satisfying the inequalities 
\[
w_j \chi < \chi_j < w_j \chi + l- n, 
\]
for $j \in \{ 1, \dots , N-1 \}$, where $\chi = \chi(M_{E, V})$ and $\chi_j = \chi (M_{E, V} \otimes \mathcal{O}_{C_j})$. Then since $M_{E, V} \otimes \mathcal{O}_{C_j}$ 
is semistable for each $j$, by \cite[Theorem 1, Step-2]{Bigas91}, there exists a polarization $w= (w_1, \dots , w_N)$ such that $M_{E, V}$ is $w$-semistable. So, to 
complete the proof, we prove the existence of a polarization $w$ satisfying the above mentioned inequalities. As $M_{E, V} \otimes \mathcal{O}_{C_j}$ is a subbundle 
of the trivial bundle $V \otimes \mathcal{O}_{C_j}$ for each $j$, we have $\chi_j < 0$ for each $j$. Also, since $\chi = \sum\limits_{j =1}^N \chi_j - (l-n)(N-1)$, we conclude that for each $j$, $0 < \frac{\chi_j}{\chi} <1.$

Now, we will first prove the existence of $w_1 \in (0, 1) \cap \mathbb{Q}$ such that 
\begin{equation}\label{inequalities for w1}
        \begin{cases}
         \chi_1 < w_1 \chi + l-n \\
         \chi_1 > w_1 \chi.
        \end{cases}
  \end{equation}
  The inequalities (\ref{inequalities for w1}) are equivalent to
  \begin{equation}\label{another form of inequalities for w1}
    \chi_1 - (l-n) < w_1  \chi < \chi_1.  
  \end{equation}
  Multiplying the inequalities (\ref{another form of inequalities for w1}) by $\frac{1}{\chi}$ and using the hypothesis that $\chi < 0$, we get 
  \begin{equation*}\label{inequalities for w0}
     \frac{\chi_1}{\chi} < w_1 < \frac{\chi_1 - (l-n)}{\chi}.
  \end{equation*}
We know that both $\frac{\chi_1 - (l-n)}{\chi}$ and $\frac{\chi_1}{\chi}$ are rational numbers lying between 0 and 1. So it is always possible to choose $w_1 \in (0,1) \cap \mathbb{Q}$ satisfying inequalities (\ref{inequalities for w1}).
In exactly the same way, for each $j \in \{1,\dots,N-1\}$,  we can choose $w_j$ such that  
\begin{equation*}\label{inequalities for wj}
     \frac{\chi_j}{\chi} < w_j <\frac{\chi_j - (l-n)}{\chi}.
\end{equation*}
Adding all these inequalities, we obtain
    \begin{equation*}
        \frac{\sum\limits_{j=1}^{N-1}\chi_j}{\chi} < \sum\limits_{j=1}^{N-1} w_j  < \frac{\sum\limits_{j=1}^{N-1}\chi_j - (l-n)(N-1)}{\chi}.  
    \end{equation*}
   This implies $\sum\limits_{j=1}^{N-1} w_j \in (0, 1) \cap \mathbb{Q}$ as well. Finally, defining $w_N = 1 - \sum\limits_{j = 1}^{N-1} w_j$, we get the required polarization.
\end{proof} 
\end{theorem}

Now suppose that $(L,V)$ is a generated pair where $L$ is a line bundle on $C$ and $(L_j,V_j)$ is a general linear series for each $j$. Then it is well known that $M_{L_j,V_j}$ is semistable.
\begin{theorem}[{\cite[Theorem 5.1]{Usha15}}] \label{Butler's conjecture for line bundles}
    Let $X$ be a smooth general curve of genus $g \geq 1$ and $(L, V)$ a general linear series of type $(d, n+1)$ (i.e., ${\rm deg}(L) = d$, ${\rm dim}(V) = n+1$ and $V$ generates $L$). 
    Then $M_{L,V}$ is semistable.
\end{theorem}
 We provide a complete characterization of kernel bundles associated to the generated pair $(L,V)$ on $C$ that are semistable with respect to a polarization $w$ under some mild assumptions on the degrees of $L_j$.
\begin{theorem}\label{thm: about (L,V)}
Let $C$ be a comb-like curve whose components are general members of their corresponding moduli spaces. Suppose that $L$ is a globally generated line bundle on $C$ and $(L, V)$ is a generated pair such that $(L_j, V_j)$ is a general linear series for each $j$. Then we have the following:
\begin{enumerate}
\item[(i)] If $\text{rank}(M_{L, V}) = 2$, then there exists a polarization $w$ such that $M_{L, V}$ is $w$-semistable if and only if $M_{L, V} \otimes \mathcal{O}_{C_j} \cong M_{L_j, V_j}$ for each $j \in \{1,\dots,N\}$.
\item[(ii)] Let $\text{rank}(M_{L,V}) = l-1 > 2$. 
\begin{enumerate}
    \item If $l -1$ divides $d_j$ for each $j$ $(\text{where} \; d_j = \deg(L_j))$, then there exists a polarization $w$ such that $M_{L, V}$ is $w$-semistable if and only if $M_{L, V} \otimes \mathcal{O}_{C_j} \cong M_{L_j, V_j}$.
    \item If $d_j \neq l-1-r_j$ for each $j \in \{1,\dots,N-1\}$, then there exists a polarization $w$ such that $M_{L, V}$ is $w$-semistable if and only if $M_{L, V} \otimes \mathcal{O}_{C_j} \cong M_{L_j, V_j}$.
\end{enumerate}

\end{enumerate}
\begin{proof}
Suppose that $M_{L, V} \otimes \mathcal{O}_{C_j} \cong M_{L_j, V_j}$ for each $j$. By Theorem \ref{Butler's conjecture for line bundles}, 
$M_{L_j, V_j}$ is semistable on $C_j$ as $(L_j, V_j)$ is a general linear series. This implies $M_{L,V} \otimes \mathcal{O}_{C_j}$ is semistable for each $j$. So, by Theorem \ref{thm: existence of polarization}, 
there exists a polarization $w$ such that $M_{L, V}$ is $w$-semistable. This proves the forward implication for both the cases.

Conversely, suppose that $M_{L, V}$ is $w$-semistable. 
\begin{enumerate}
    \item[(i)] If $M_{L, V} \otimes \mathcal{O}_{C_j} \ncong M_{L_j, V_j}$ for some $j$, then $\ker ({\rho_j}_{\mid_V}) \neq 0$ for some $j \in \{1,\dots,N-1\}$. This implies by (i) of Theorem \ref{cor:  on strongly unstable} that $M_{L,V}$ is strongly unstable, which is a contradiction.
   
    \item[(ii)(a)] As $l-1$ divides $d_j$ for each $j$, it forces $\ker ({\rho_j}_{\mid_V})$ to be zero, by Remark \ref{crucial remark after Proposition 3.1}. Therefore, $M_{L, V} \otimes \mathcal{O}_{C_j} \cong M_{L_j, V_j}$ for each $j$.
    \item[(ii)(b)] Again, if $M_{L, V} \otimes \mathcal{O}_{C_j} \ncong M_{L_j, V_j}$ for some $j$, then $\ker ({\rho_j}_{\mid_V}) \neq 0$ for some $j \in \{1,\dots,N-1\}$. This, along with the fact that $d_j \neq l-1-r_j$ implies by (ii) of Theorem \ref{cor:  on strongly unstable} that $M_{L,V}$ is strongly unstable, which is a contradiction.
\end{enumerate}
\end{proof}
\end{theorem}

For the higher rank case, we obtain the similar results as in Theorem \ref{thm: about (L,V)} provided Butler's conjecture (see Conjecture \ref{Butler-conjecture}) is true.
 \begin{theorem} \label{The iff condition}
 Let $C$ be as in Theorem \ref{thm: about (L,V)}. Let $(E_i,V_i)$ be a general choice of a generated pair on $C_i$ 
 with $E_i$ semistable of rank $n$ for each $i$, and $E$ be a vector bundle on $C$ obtained by gluing $E_i$'s. Then there exists a linear subspace $V \subseteq H^0(E)$ such that $\dim V = l > n$ and $(E,V)$ is a generated pair on $C$. Further, suppose that the Butler's conjecture is true. Then the following are true:
     \begin{enumerate}
         \item[(i)] If $\text{rank}(M_{E,V}) =2$, then there exists a polarization $w$ such that $M_{E,V}$ is $w$-semistable if and only if ${M_{E,V}} \otimes \mathcal{O}_{C_j} \cong M_{E_j,V_j}$ for each $j$.
                 \item[(ii)] If $\text{rank}(M_{E,V}) = (l-n) > 2$, and
         \begin{enumerate}
             \item If $l-n$ divides $d_j$ for each $j$, then there exists a polarization $w$ 
         such that $M_{E,V}$ is $w$-semistable if and only if ${M_{E,V}} \otimes \mathcal{O}_{C_j} \cong M_{E_j,V_j}$ for each $j$.
         \item If $d_j \neq l-n-r_j$ for each $j \in \{1,\dots,N-1\}$, then there exists a polarization $w$ 
         such that $M_{E,V}$ is $w$-semistable if and only if ${M_{E,V}} \otimes \mathcal{O}_{C_j} \cong M_{E_j,V_j}$ for each $j \in \{1,\dots,N\}$.
         \end{enumerate}
     \end{enumerate}
\begin{proof}
 The existence of a linear subspace $V \subseteq H^0(E)$ with the desired properties follows directly from the way $E$ is constructed out of $E_i$'s. \\
 Now, for the remaining parts of the proof, suppose that the Butler's conjecture is true.
If ${M_{E,V}} \otimes \mathcal{O}_{C_j} \cong M_{E_j, V_j}$ for each $j$, then since $M_{E_j, V_j}$ is semistable (by Butler's conjecture), we conclude that ${M_{E,V}} \otimes \mathcal{O}_{C_j}$ is semistable. Therefore, there exists a polarization $w$ such that $M_{E, V}$ is $w$-semistable (by Theorem \ref{thm: existence of polarization}).

Conversely, suppose that $M_{E,V}$ is $w$-semistable. If ${M_{E,V}} \otimes \mathcal{O}_{C_j} \ncong M_{E_j,V_j}$ for some $j$, then $\ker ({\rho_j}_{\mid_V}) \neq 0$ for some $j \in \{1,\dots,N-1\}$. The results now follow from Theorem \ref{cor:  on strongly unstable}.
\end{proof}
 \end{theorem}

\section*{Acknowledgements}
	We thank the referee for carefully reading the article and for all the suggestions which have improved the overall quality of the article.

	The first and the third authors were partially supported by a grant from Infosys Foundation. The third author was supported by the National Board for Higher Mathematics (NBHM), Department of Atomic Energy, Government of India (0204/2/2022/R\&D-II/1785).
	
\end{document}